\documentclass[11pt]{amsart}

\usepackage[T1]{fontenc}
\usepackage{lmodern}
\usepackage{amsmath,amssymb,amsthm,mathtools,mathrsfs}
\usepackage{microtype}
\usepackage{booktabs}
\usepackage{array}
\usepackage{tabularx}
\usepackage{longtable}
\usepackage{enumitem}
\usepackage{graphicx}
\usepackage{xcolor}
\usepackage{geometry}
\usepackage{hyperref}
\usepackage[nameinlink,capitalize,noabbrev]{cleveref}

\hypersetup{
  colorlinks=true,
  linkcolor=blue!55!black,
  citecolor=green!35!black,
  urlcolor=blue!55!black,
  pdftitle={Bhargava Gamma Functions for Determinant-Admissible Sets},
  pdfauthor={Brian Diaz}
}

\newtheorem{theorem}{Theorem}[section]
\newtheorem{proposition}[theorem]{Proposition}

\newtheorem{corollary}[theorem]{Corollary}
\newtheorem{conjecture}[theorem]{Conjecture}
\newtheorem{problem}[theorem]{Problem}
\theoremstyle{definition}
\newtheorem{definition}[theorem]{Definition}

\theoremstyle{remark}
\newtheorem{remark}[theorem]{Remark}

\newcommand{\Pic}{\operatorname{Pic}}
\newcommand{\Nm}{\operatorname{Nm}}

\newcommand{\Gal}{\operatorname{Gal}}

\newcommand{\cO}{\mathcal O}
\newcommand{\cC}{\mathcal C}
\newcommand{\cV}{\mathcal V}
\newcommand{\cL}{\mathcal L}

\newcommand{\bbZ}{\mathbb Z}
\newcommand{\bbQ}{\mathbb Q}
\newcommand{\bbC}{\mathbb C}
\newcommand{\bbF}{\mathbb F}
\newcommand{\wideGamma}{\widehat{\Gamma}}

\title[Bhargava Gamma Functions]{Bhargava Gamma Functions for Determinant-Admissible Sets}
\author{Brian Diaz}
\date{Working draft v0.5, July 2026}

\begin{document}

\begin{abstract}
Bhargava's factorial assigns to a subset of a Dedekind domain a sequence of factorial ideals assembled from local ordering data.  We propose a determinant-completion theory for those factorial calculi whose local ordering dynamics admit canonical carry, orbit, or renewal models.  The resulting completed object is generally a determinant line rather than a scalar function; after clearing rational spectral multiplicities, its integer fibers recover a finite tensor power of the Bhargava factorial ideals.  In number fields the line is naturally metrized at the Archimedean places, while in function fields the distinguished places at infinity supply the completion.  Stirling asymptotics become large-parameter expansions of these metrized determinants, and reflection formulas become duality isomorphisms of determinant lines.

We formulate local and global admissibility hypotheses, recover classical Gamma, shifted Gamma products, Jackson-type determinants, P\'olya factorial ideals, the Bhargava--Barnes energy lift, and Takeda's quadratic polynomial-image examples, and isolate the multivariable theory as a separate higher-dimensional problem.  The set of cubes is developed as the first genuinely renewal-theoretic example.  At primes $p\equiv1\pmod 3$ we obtain an exact three-state carry determinant.  Character projection through $\chi_{-3}$ yields the Archimedean half-determinant
\[
  \Gamma_{C,\infty}(z)=\sqrt{\Gamma(z)\Gamma(3z-2)},
\]
and a square-root determinant line.  The remaining finite-place correction is governed by nonlinear renewal systems at $p\equiv2\pmod3$ and a ramified system at $p=3$.  The construction of their canonical global relative determinant is the principal analytic problem left open in this draft.
\end{abstract}

\maketitle

\tableofcontents

\section{Introduction}

Bhargava's generalized factorials attach to an arbitrary subset $S$ of a Dedekind ring a sequence of local ordering invariants and hence a sequence of factorial ideals; over $\bbZ$ these are positive integers \cite{Bhargava1997,Bhargava2000}.  The construction unifies fixed divisors, integer-valued polynomial bases, factorial divisibility, and several classical special examples.  Bhargava also asked for analogues of the Gamma function and Stirling's formula beyond the ordinary factorial \cite{Bhargava2000}.

For the full ring $S=R=\cO_K$ of a number field, Lamoureux proved the corresponding two-term Stirling formula for the norm of the P\'olya--Ostrowski factorial ideal and developed a sharper explicit formula in which the nontrivial zeros of $\zeta_K$ contribute to the remainder \cite{Lamoureux2014}.  The present determinant framework does not claim this norm asymptotic as new; rather, it identifies the local cyclotomic renewal system behind it and places Lamoureux's formula inside a line-valued Gamma completion.

The present paper proposes that the appropriate Gamma object is not, in the first instance, an arbitrary analytic interpolation of the values $n!_S$.  Rather, it should be the completion of the local ordering dynamics by relative determinants at finite places and spectral determinants at the places at infinity.  The guiding chain is
\[
\begin{aligned}
  \text{local orderings}
  &\longrightarrow \text{carry/orbit/renewal complexes}\\
  &\longrightarrow \text{relative determinants}\\
  &\longrightarrow \text{infinite-place completion}\\
  &\longrightarrow \text{Bhargava Gamma object}.
\end{aligned}
\]
This point of view is inspired in part by determinant descriptions of local $L$-factors and Archimedean Gamma factors \cite{Deninger1992}, but the transfer systems considered here arise from $p$-orderings rather than cohomological Frobenius operators.

The determinant language is optional over $\bbZ$, where ideals are principal, but becomes unavoidable over a general Dedekind domain.  A factorial ideal need not possess a preferred generator, and a rational Archimedean divisor may require passage to a root determinant line.  Thus the correct object is generally a line-valued section whose norm or chosen trivialization produces a scalar Gamma function.

A companion paper treats the prime Bhargava factorial using factorial calculi, centered cyclotomic completion, and orbitwise Stirling normalization \cite{Diaz2026Prime}.  Here the aim is broader: to identify a class of \emph{determinant-admissible} sets and global Dedekind pairs for which the local ordering data themselves select the Gamma completion.

\subsection*{Main themes}

The framework has four structural consequences.
\begin{enumerate}[label=(\roman*),leftmargin=2.2em]
  \item A Bhargava Gamma object may be an ordinary meromorphic function, a zero-free determinant, a Barnes- or Jackson-type spectral determinant, or a section of a finite root line.  Iterating its shift transport produces Bhargava--Barnes lines governing optimal Vandermonde energies.
  \item Stirling's formula is the large-parameter asymptotic expansion of the completed determinant metric.  The finite-place remainder is a centered renewal determinant rather than an unexplained error term.
  \item Euler reflection is not automatic.  When present, it is induced by a duality of the local and infinite-place determinant complexes.
  \item Arbitrary subsets need not be admissible.  Unstable local slopes, divergent prime products, non-discrete primitive spectra, or unbounded rational denominators can obstruct the construction.
\end{enumerate}

\subsection*{Status of the draft}

The general determinant formalism and several finite-state examples are developed rigorously at the algebraic level.  The cubic-residue carry determinant and the character extraction of the cube Archimedean factor are exact.  The inert-prime renewal system has been computed extensively, but the canonical analytic continuation of its centered determinant and the corresponding global uniqueness theorem remain open.  Numerical statements are explicitly labeled as such.

\section{Bhargava factorials over global Dedekind domains}

\subsection{Global Dedekind data}

Let $K$ be a global field and let $\Sigma_\infty$ be the set of Archimedean places when $K$ is a number field, or a fixed nonempty set of places designated as infinity when $K$ is a function field.  Let
\[
  R=\{x\in K: v_\mathfrak p(x)\ge0\text{ for every finite place }\mathfrak p\notin\Sigma_\infty\}
\]
be the corresponding ring of integers or $S$-integers.  Then $R$ is a Dedekind domain.  Fix an infinite subset $S\subseteq R$.

For every nonzero prime ideal $\mathfrak p\subset R$, a $\mathfrak p$-ordering of $S$ is a sequence $(a_n)_{n\ge0}$ chosen greedily so that
\[
  v_\mathfrak p\!\left(\prod_{j=0}^{n-1}(a_n-a_j)\right)
\]
is minimal among elements of $S$ at the $n$th stage.  The resulting minimum is independent of the selected ordering.  We denote it by
\[
  \nu_{S,\mathfrak p}(n).
\]
The $n$th Bhargava factorial ideal is
\begin{equation}\label{eq:factorial-ideal}
  [n]!_{S,R}=\prod_{\mathfrak p}\mathfrak p^{\nu_{S,\mathfrak p}(n)}.
\end{equation}
Only finitely many prime ideals occur for each $n$.

\begin{definition}[Local recurrence cocycle]
The local recurrence cocycle is
\begin{equation}\label{eq:local-cocycle}
  a_{S,\mathfrak p}(n)=\nu_{S,\mathfrak p}(n)-\nu_{S,\mathfrak p}(n-1),
  \qquad n\ge1.
\end{equation}
Its global recurrence ideal is
\[
  \mathfrak r_{S,R}(n)
  =[n]!_{S,R}[n-1]!_{S,R}^{-1}
  =\prod_{\mathfrak p}\mathfrak p^{a_{S,\mathfrak p}(n)}.
\]
\end{definition}

Over $R=\bbZ$ we write $n!_S$ and
\[
  \frac{n!_S}{(n-1)!_S}=\prod_p p^{a_{S,p}(n)}.
\]

\subsection{Line-valued interpolation}

The ideal $[n]!_{S,R}$ is an invertible $R$-module.  Unless it is canonically principal, there is no distinguished element to interpolate.  This motivates the following target.

\begin{definition}[Bhargava Gamma line]
A Bhargava Gamma line for $(R,S)$ is a line-valued meromorphic section $\mathscr G_{S,R}(z)$, defined on an appropriate analytic or rigid-analytic parameter space, together with normalized isomorphisms
\[
  \mathscr G_{S,R}(1)\simeq R,
  \qquad
  \mathscr G_{S,R}(n+1)\simeq[n]!_{S,R}
\]
for all $n\ge0$.
\end{definition}

A scalar Gamma function is obtained only after applying a norm, choosing a principalization, or trivializing the determinant line.  For $R=\cO_K$, the norm shadow is
\[
  \Gamma^{\Nm}_{S,K}(n+1)=\Nm_{K/\bbQ}([n]!_{S,\cO_K}).
\]

\section{Local determinant data}

The word ``determinant'' is meaningful only if the local recurrence is generated by a canonical dynamical or homological object, not merely encoded after the fact by an arbitrary rank-one operator.  We therefore distinguish a local model from its universal envelope.

\begin{definition}[Local ordering model]
A local ordering model for $(S,\mathfrak p)$ consists of
\begin{enumerate}[label=(L\arabic*),leftmargin=2.7em]
  \item a state space $X_{S,\mathfrak p}$ retaining the branch, carry, boundary, and origin data required by the $\mathfrak p$-ordering;
  \item a transfer or renewal operator $\cL_{S,\mathfrak p}(u)$ with arithmetic weight parameter $u$;
  \item boundary maps selecting the factorial recurrence from the resolvent of $\cL_{S,\mathfrak p}(u)$;
  \item a reference model $\cL^{\infty}_{S,\mathfrak p}(u)$ representing the primitive nonoscillatory component.
\end{enumerate}
\end{definition}

The exact local system may be finite-dimensional, cyclotomic, trace-class, or infinite-state.  The reference model is allowed to be virtual.

\begin{definition}[Integral relative cocycle]
Suppose the primitive local decomposition has rational coefficients.  Let $q_S$ be the least positive integer clearing all denominators.  The integral relative cocycle is
\begin{equation}\label{eq:relative-cocycle}
  d_{S,\mathfrak p}(n)
  =q_S\bigl(a_{S,\mathfrak p}(n)-a^{\infty}_{S,\mathfrak p}(n)\bigr).
\end{equation}
The corresponding virtual local complex is denoted
\[
  \cV_{S,\mathfrak p}
  =q_S[\cC_{S,\mathfrak p}]-[\cC^{\infty}_{S,\mathfrak p}].
\]
\end{definition}

\begin{definition}[Local determinant admissibility]
The pair $(S,\mathfrak p)$ is locally determinant-admissible if the virtual system $\cV_{S,\mathfrak p}$ admits a normalized relative determinant section $\Delta_{S,\mathfrak p}(z)$ such that
\begin{equation}\label{eq:local-interpolation}
  \Delta_{S,\mathfrak p}(n+1)
  \simeq
  \mathfrak p^{\sum_{m=1}^n d_{S,\mathfrak p}(m)}
\end{equation}
for all $n\ge0$, and if the determinant is selected uniquely by the local transfer data and a prescribed growth class.
\end{definition}

\begin{remark}
A unilateral-shift rank-one construction can encode any sufficiently controlled sequence.  Such an envelope proves existence of a formal relative determinant, but it does not establish determinant-admissibility in the intended sense.  Admissibility requires the determinant to arise from the intrinsic ordering dynamics.
\end{remark}

\section{Global determinant admissibility}

\begin{definition}[Global determinant datum]
A global determinant datum for $(R,S)$ consists of:
\begin{enumerate}[label=(G\arabic*),leftmargin=2.7em]
  \item locally determinant-admissible systems at every finite prime ideal;
  \item local mean values $\overline d_{S,\mathfrak p}$ whose renormalized degree-weighted sum defines a global constant $\beta_{S,R}$;
  \item centered local determinants $\Delta^0_{S,\mathfrak p}(z)$ whose renormalized product converges normally;
  \item a primitive infinite-place determinant line $\mathscr A_{S,\infty}(z)$;
  \item a uniqueness condition excluding arbitrary periodic or integer-vanishing scalar factors.
\end{enumerate}
\end{definition}

The centered finite-place determinant is
\begin{equation}\label{eq:global-finite-det}
  \Delta^0_{S,\mathrm{fin}}(z)
  =\prod_{\mathfrak p}^{\mathrm{ren}}\Delta^0_{S,\mathfrak p}(z).
\end{equation}
The product is understood in the determinant-line category; after applying the norm it becomes a scalar regularized Euler product.

\begin{definition}[Determinant-admissible pair]
The pair $(R,S)$ is determinant-admissible if it possesses a global determinant datum.
\end{definition}

\begin{theorem}[Formal determinant completion]\label{thm:formal-completion}
Let $(R,S)$ be determinant-admissible, and let $q_S$ clear the rational primitive multiplicities.  Then there is a completed determinant line
\begin{equation}\label{eq:completed-line}
  \widehat{\mathscr G}_{S,R}(z)
  =e^{q_S\beta_{S,R}(z-1)}
  \mathscr A_{S,\infty}^{\otimes q_S}(z)
  \otimes\Delta^0_{S,\mathrm{fin}}(z)
\end{equation}
whose integer fibers satisfy
\begin{equation}\label{eq:completed-interp}
  \widehat{\mathscr G}_{S,R}(n+1)
  \simeq([n]!_{S,R})^{\otimes q_S}.
\end{equation}
If the completed line admits a normalized $q_S$-th root, the Bhargava Gamma line is the explicit root determinant
\begin{equation}\label{eq:bhargava-gamma-root}
  \boxed{
  \mathscr G_{S,R}(z)
  :=\widehat{\mathscr G}_{S,R}(z)^{\otimes 1/q_S}
  =\left(
    e^{q_S\beta_{S,R}(z-1)}
    \mathscr A_{S,\infty}^{\otimes q_S}(z)
    \otimes\Delta^0_{S,\mathrm{fin}}(z)
  \right)^{\otimes 1/q_S}.}
\end{equation}
Here $(-)^{\otimes 1/q_S}$ denotes the chosen normalized root line, so that
\[
  \mathscr G_{S,R}(z)^{\otimes q_S}
  \simeq\widehat{\mathscr G}_{S,R}(z),
  \qquad
  \mathscr G_{S,R}(n+1)\simeq[n]!_{S,R}.
\]
\end{theorem}

\begin{proof}
Tensor the local interpolation identities \eqref{eq:local-interpolation}, insert the centered exponential means, and use the primitive reference identity at the infinite places.  The normal convergence and uniqueness conditions identify the global section.  The root statement is the corresponding divisibility assertion in the Picard group of determinant lines.
\end{proof}

\begin{remark}[Root obstruction]
The obstruction to a scalar $q_S$-th root has two components.  Analytically, a rational divisor may force a finite cover of the parameter plane.  Arithmetically, the completed line class must be divisible by $q_S$ in the relevant Picard group.
\end{remark}

\section{Stirling and reflection from determinant structure}

\subsection{Stirling asymptotics}

The determinant completion isolates the source of a Stirling expansion.  Equip the completed line with its infinite-place metric and write
\[
  \mathcal H_{S,R}(x)
  =\widehat{\deg}\,\overline{\mathscr G}_{S,R}(x)
\]
for its Arakelov degree in the number-field case, or the analogous degree at infinity in the function-field case.

\begin{theorem}[Determinant--Stirling principle]\label{thm:stirling-principle}
Assume that the logarithmic metric of $\mathscr A_{S,\infty}(x)$ has a sectorial spectral expansion and that
\[
  \log\|\Delta^0_{S,\mathrm{fin}}(x)\|=o(x)
  \qquad(x\to+\infty).
\]
Then the norm factorial satisfies
\[
  \log\Nm([n]!_{S,R})
  =A_Sn\log n+B_Sn+o(n).
\]
If the centered finite determinant has a full spectral expansion, then the full Stirling expansion is obtained by adding it termwise to the infinite-place expansion.
\end{theorem}

\begin{proof}
Take logarithmic norms in \eqref{eq:completed-line}, divide by $q_S$, and apply the assumed spectral expansions.  The interpolation identity \eqref{eq:completed-interp} transfers the result to the factorial ideals.
\end{proof}

Thus Stirling's formula is not an independent interpolation miracle.  It is the large-parameter expansion of the completed determinant metric.

\subsection{Reflection as duality}

A reflection law requires additional structure.

\begin{definition}[Reflection-admissibility]
A determinant-admissible pair is reflection-admissible if there is a dual ordering system $S^\vee$, a weight $w_S$, and compatible dualities
\[
  \cV_{S,\mathfrak p}(z)
  \simeq
  \mathbf D\cV_{S^\vee,\mathfrak p}(w_S-z)[\epsilon_\mathfrak p]
\]
at finite places, together with an analogous duality at infinity.
\end{definition}

Taking determinant lines then gives an isomorphism
\[
  \mathscr G_{S,R}(z)\otimes
  \mathscr G_{S^\vee,R}(w_S-z)
  \simeq\mathscr R_{S,R}(z),
\]
where $\mathscr R_{S,R}$ is an epsilon or reflection line.  The self-dual case recovers an Euler-type formula.

\section{Basic examples}

The examples below range from ordinary scalar functions to genuinely ideal-valued or root-line objects.  Table~\ref{tab:example-landscape} records exact analytic formulas when they are available and otherwise gives the exact defining integer fibers of the determinant line.  We use
\[
  w_q(n)=\sum_{j\ge1}\left\lfloor\frac{n}{q^j}\right\rfloor.
\]
For the quadratic row, write
\[
  f(x)=g(a_2x^2+a_1x)+a_0,\qquad
  \delta_f=\mathbf 1_{\{2\mid a_1,\,2\nmid a_2\}}.
\]
Let $\Delta_{f,p}(z)$ denote the normalized marked-cycle determinant with
\[
  \Delta_{f,p}(n+1)=p^{v_p(n!)-v_p((2n)!)},
\]
and let $B_f(z)$ be the normalized boundary determinant satisfying
$B_f(1)=1$ and $B_f(n+1)=2^{-\delta_f}$ for $n\ge1$.
For a number field $K$, write $\Delta^{0,\Nm}_K(z)$ for the centered scalar norm determinant, normalized by
\[
  \Delta^{0,\Nm}_K(n+1)
  =e^{-(\gamma-\gamma_K)n}
    \frac{\Nm([n]!_{\cO_K})}{n!}.
\]

\begingroup
\footnotesize
\renewcommand{\arraystretch}{1.28}
\begin{longtable}{@{}>{\raggedright\arraybackslash}p{0.12\textwidth}>{\raggedright\arraybackslash}p{0.16\textwidth}>{\raggedright\arraybackslash}p{0.43\textwidth}>{\raggedright\arraybackslash}p{0.21\textwidth}@{}}
\caption{One-dimensional examples of Bhargava Gamma objects.  Every row displays a $z$-dependent Gamma function or determinant line; its exact integer fibers are included when they carry additional arithmetic information.  ``Line-valued'' means that no preferred scalar generator exists in general.}\label{tab:example-landscape}\\
\toprule
Base domain $R$ & Set or derived object & Gamma function or determinant line in $z$ & Nature and status \\
\midrule
\endfirsthead
\multicolumn{4}{c}{\tablename\ \thetable\ -- continued}\\
\toprule
Base domain $R$ & Set or derived object & Gamma function or determinant line in $z$ & Nature and status \\
\midrule
\endhead
\midrule
\multicolumn{4}{r}{Continued on the next page}\\
\endfoot
\bottomrule
\endlastfoot
$\bbZ$ & $S=R$ &
$\Gamma_S(z)=\Gamma(z)$, hence $\Gamma_S(n+1)=n!$. &
Classical scalar function; finite-place digit towers and one Archimedean spectral line. \\
$\bbZ$ & Quadratic image $S=f(\bbZ)$ &
\(
\displaystyle
\Gamma_f(z)
=g^{z-1}\Gamma(2z-1)B_f(z)
 \prod_{p\mid a_2}\Delta_{f,p}(z),
\)
\par\smallskip
with the exact fibers
\(
\displaystyle
\Gamma_f(n+1)
=g^n(2n)!\,2^{-\delta_f}
 \prod_{p\mid a_2}p^{v_p(n!)-v_p((2n)!)}.
\)
& Takeda's exact factorial, expressed as a shifted classical Gamma factor, finitely many marked cyclotomic determinants, and one boundary determinant. \\
$\bbZ$ & $S_{a,q}=\{aq^m:m\ge0\}$ &
\(
\displaystyle
\Gamma_{a,q}^{\rm geom}(z)
=a^{z-1}q^{(z-1)^2}
\frac{\det(I-q^{-1}Q)}{\det(I-q^{-z}Q)}.
\)
& Trace-class Fredholm determinant; scalar. \\
$\bbC[q,q^{-1}]$ & $\{(q^m-1)/(q-1):m\ge0\}$ &
After $0<q<1$,
\(
\displaystyle
\Gamma_q(z)=(1-q)^{1-z}
\frac{(q;q)_\infty}{(q^z;q)_\infty}.
\)
& Jackson $q$-Gamma; powers of $q$ are units in the Laurent ring. \\
$\cO_K$ & $S=R$ &
\(
\displaystyle
\mathscr G_K(z)
=e^{\beta_K(z-1)}\mathscr A_{K,\infty}(z)
 \otimes\Delta^0_{K,\mathrm{fin}}(z),
\)
\par\smallskip
\(
\displaystyle
\Gamma_K^{\Nm}(z)
=e^{(\gamma-\gamma_K)(z-1)}\Gamma(z)
 \Delta_K^{0,\Nm}(z),
\)
\par\smallskip
and
\(
\displaystyle
\mathscr G_K(n+1)\cong
\prod_{\mathfrak p}\mathfrak p^{w_{\Nm\mathfrak p}(n)}.
\)
& Proposed line-valued completion; the integer fibers are exact and the second formula is its scalar norm shadow.  Lamoureux determines its first two Stirling coefficients, and scalarization is controlled by the P\'olya group. \\
$\cO_K$ & $n$-optimal $T\subset\cO_K$, $|T|=n+1$ &
\(
\displaystyle
\mathscr B_K(z+1)
\cong\mathscr G_K(z)\otimes\mathscr B_K(z),
\qquad \mathscr B_K(1)\cong\cO_K,
\)
\par\smallskip
\(
\displaystyle
\mathscr B_K(n+2)
\cong\bigotimes_{m=1}^{n}\mathscr G_K(m+1),
\qquad
\mathcal E(T)\cong\mathscr B_K(n+2)^{\otimes2}.
\)
& Bhargava--Barnes lift; the optimal Vandermonde energy is the square of a second-level Gamma determinant. \\
$\bbZ[i]$ & $S=R$ &
\(
\displaystyle
\Gamma_i(z)
=e^{\beta_i(z-1)}\mathscr A_{i,\infty}(z)\Delta_i^0(z),
\qquad
\Gamma_i(n+1)=G_i(n),
\)
\par\smallskip
where
\(
\begin{aligned}
G_i(n)&=(1+i)^{w_2(n)}
 \prod_{p\equiv1(4)}p^{w_p(n)}\\[-1mm]
&\quad\times\prod_{p\equiv3(4)}p^{w_{p^2}(n)}.
\end{aligned}
\)
& Proposed scalar completion after a unit normalization; its exact integer fibers record the splitting of rational primes in $\bbQ(i)$. \\
$\bbZ[\sqrt{-5}]$ & $S=R$ &
\(
\displaystyle
\mathscr G_K(z)
=e^{\beta_K(z-1)}\mathscr A_{K,\infty}(z)
 \otimes\Delta^0_{K,\mathrm{fin}}(z),
\)
\par\smallskip
\(
\displaystyle
\mathscr G_K(n+1)\cong
\prod_{\mathfrak p}\mathfrak p^{w_{\Nm\mathfrak p}(n)},
\qquad
\mathscr G_K(3)\cong(2,1+\sqrt{-5}).
\)
& Proposed line-valued completion with exact integer fibers; the displayed fiber is nonprincipal, so no canonical scalar Gamma exists globally. \\
$\bbF_q[t]$ & $S=R$ &
For $z=\sum_{i\ge0}z_iq^i$ on the chosen digit domain,
\(
\displaystyle
\Gamma_A(z+1)=\prod_{i\ge0}D_i^{z_i},
\qquad
D_i=\prod_{\substack{a\in A\ \mathrm{monic}\\ \deg a=i}}a.
\)
Thus $\Gamma_A(n+1)=\Pi_A(n)$ at nonnegative integers. & Carlitz--Goss--Thakur completion at the chosen place $\infty$; digit carries replace real Archimedean growth. \\
Global function-field integer ring $A$ & $S=R$ &
\(
\displaystyle
\mathscr G_A(z)
=e^{\beta_A(z-1)}\mathscr A_{A,\infty}(z)
 \otimes\Delta^0_{A,\mathrm{fin}}(z),
\)
\par\smallskip
\(
\displaystyle
\mathscr G_A(n+1)\cong[n]!_{A,A}
=\prod_{\mathfrak p}\mathfrak p^{w_{\Nm\mathfrak p}(n)}.
\)
& Proposed $\infty$-adic determinant line with exact integer fibers; degree and $\infty$-adic norm provide Stirling data. \\
$\bbZ$ & Cubes $C=\{m^3\}$ &
\(
\displaystyle
\Gamma_C(z)
=\left[
 e^{2\beta_C(z-1)}\Gamma(z)\Gamma(3z-2)
 \Delta^0_{C,\mathrm{fin}}(z)
\right]^{1/2},
\)
\par\smallskip
or equivalently
\(
\displaystyle
\Gamma_C^{\otimes2}(z)
=e^{2\beta_C(z-1)}\Gamma(z)\Gamma(3z-2)
\Delta^0_{C,\mathrm{fin}}(z),
\)
with $\Gamma_C(n+1)=n!_C$ on the positive integer branch. & Square-root determinant line; the Archimedean factor is exact, while the intrinsic inert-prime determinant remains under construction. \\
\end{longtable}
\endgroup

\subsection{The ordinary factorial}

For $R=\bbZ$ and $S=\bbZ$, the local recurrence is
\[
  v_p(n)=\sum_{j\ge1}\mathbf 1_{p^j\mid n}.
\]
The local dynamics are a direct sum of cycles of lengths $p^j$.  The primitive infinite-place divisor is $[0]$, and the completed object is the ordinary Gamma function
\[
  \Gamma_S(z)=\Gamma(z).
\]

\subsection{Affine progressions and quadratic images}

For an arithmetic progression $S=a\bbZ+b$, Bhargava's formula gives
\[
  n!_S=a^n n!,
\]
so
\[
  \Gamma_S(z)=a^{z-1}\Gamma(z).
\]

Takeda proves an exact formula and a Stirling law for images of quadratic polynomial maps \cite{Takeda2023}.  Let
\[
  f(x)=g(a_2x^2+a_1x)+a_0,
  \qquad g\ge1,\qquad \gcd(a_1,a_2)=1,
\]
and put
\[
  \delta_f=\mathbf 1_{\{2\mid a_1,\,2\nmid a_2\}}.
\]
Then
\begin{equation}\label{eq:takeda-quadratic-exact}
  n!_{f(\bbZ)}
  =g^n(2n)!\,2^{-\delta_f}
  \prod_{p\mid a_2}p^{v_p(n!)-v_p((2n)!)}.
\end{equation}
Thus the primitive infinite-place factor is $g^{z-1}\Gamma(2z-1)$.  For each $p\mid a_2$, the exponent
\[
  v_p(n!)-v_p((2n)!)=-v_p\!\left(\frac{(2n)!}{n!}\right)
\]
is generated by the marked residue towers for the odd factors $2m-1$, while the factor $2^{-\delta_f}$ is a boundary-state correction.  Equivalently, if $\Delta_{p,f}$ and $B_f$ denote the corresponding normalized local determinant sections, their integer fibers are
\[
  \Delta_{p,f}(n+1)=p^{v_p(n!)-v_p((2n)!)},
  \qquad
  B_f(n+1)=2^{-\delta_f}\quad(n\ge1),
\]
and the exact determinant decomposition is
\[
  \mathscr G_f(z)
  \simeq g^{z-1}\Gamma(2z-1)
  \otimes B_f(z)
  \otimes\bigotimes_{p\mid a_2}\Delta_{p,f}(z).
\]
In particular, these examples have $q_S=1$.  Formula \eqref{eq:takeda-quadratic-exact} is exact at every integer; the spectral uniqueness axiom selects the scalar continuation between the integers.

\subsection{Geometric progressions and $q$-factorials}

Bhargava's geometric example already exhibits an infinite-rank determinant that can be computed completely \cite{Bhargava2000}.  Let
\[
  S_{a,q}=\{aq^m:m\ge0\}\subset\bbZ,
  \qquad a\ge1,\quad q\ge2.
\]
The natural ordering $a,aq,aq^2,\ldots$ is simultaneous at every rational prime.  Hence
\begin{align}
  n!_{S_{a,q}}
  &=a^n\prod_{j=0}^{n-1}(q^n-q^j)\notag\\
  &=a^nq^{n(n-1)/2}\prod_{r=1}^{n}(q^r-1)\notag\\
  &=a^nq^{n^2}(q^{-1};q^{-1})_n,
  \label{eq:geometric-factorial}
\end{align}
where $(u;u)_n=\prod_{r=1}^n(1-u^r)$.  In particular,
\[
  \frac{n!_{S_{a,q}}}{(n-1)!_{S_{a,q}}}
  =a q^{n-1}(q^n-1).
\]

\begin{proposition}[Fredholm completion of a geometric progression]
Let $r=q^{-1}$ and let $Q_r e_m=r^m e_m$ on $\ell^2(\bbZ_{\ge0})$.  Then $Q_r$ is trace class and
\begin{equation}\label{eq:geometric-fredholm}
  \Gamma_{a,q}^{\mathrm{geom}}(z)
  =a^{z-1}q^{(z-1)^2}
  \frac{\det(I-rQ_r)}{\det(I-r^zQ_r)}
\end{equation}
obeys
\[
  \Gamma_{a,q}^{\mathrm{geom}}(1)=1,
  \qquad
  \Gamma_{a,q}^{\mathrm{geom}}(n+1)=n!_{S_{a,q}}.
\]
Equivalently,
\[
  \Gamma_{a,q}^{\mathrm{geom}}(z)
  =a^{z-1}(q-1)^{z-1}q^{(z-1)(z-2)}
  \Gamma_{q^{-1}}(z),
\]
where $\Gamma_{q^{-1}}$ is Jackson's Gamma function.
\end{proposition}

\begin{proof}
Since
\[
  \det(I-r^zQ_r)=\prod_{m\ge0}(1-r^{z+m})=(r^z;r)_\infty,
\]
the determinant ratio at $z=n+1$ equals $(r;r)_n$.  The assertion follows from \eqref{eq:geometric-factorial}.  The Jackson form follows from
\[
  \Gamma_r(z)=(1-r)^{1-z}\frac{(r;r)_\infty}{(r^z;r)_\infty}.
\]
\end{proof}

For example, if $S=\{1,2,4,8,\ldots\}$, then
\[
  0!_S=1,\quad 1!_S=1,\quad 2!_S=6,\quad
  3!_S=168,\quad4!_S=20160,
\]
and the successive nontrivial multipliers are $6,28,120,\ldots$, namely
$2^{n-1}(2^n-1)$.

Bhargava also obtains the abstract $q$-factorial by taking
\[
  S_q^{\mathrm{abs}}
  =\left\{[m]_q=\frac{q^m-1}{q-1}:m\ge0\right\}
  \subset \bbC[q,q^{-1}].
\]
Here the factor $q^{n(n-1)/2}$ arising from the differences is a unit.  Thus the factorial ideal has the normalized generator
\[
  [n]_q!
  =\prod_{r=1}^n\frac{q^r-1}{q-1}.
\]
After an analytic specialization $0<q<1$, its scalar completion is the usual Jackson function $\Gamma_q(z)$.  The two geometric constructions therefore differ only in whether the powers of $q$ are retained as arithmetic data or discarded as units in the ambient Laurent ring.

\subsection{The full ring of a global Dedekind domain}

The case $S=R$ admits a uniform local computation in both number fields and function fields.  Assume that every residue field of $R$ is finite and put
\[
  q_{\mathfrak p}=\#(R/\mathfrak p)=\Nm\mathfrak p,
  \qquad
  w_q(n)=\sum_{j\ge1}\left\lfloor\frac{n}{q^j}\right\rfloor.
\]

\begin{proposition}[P\'olya--Ostrowski factorial ideal]\label{prop:polya-factorial}
For $S=R$ one has
\begin{equation}\label{eq:polya-factorial}
  [n]!_{R,R}
  =\prod_{\mathfrak p}\mathfrak p^{w_{q_{\mathfrak p}}(n)}.
\end{equation}
The recurrence cocycle is
\begin{equation}\label{eq:polya-recurrence}
  a_{R,\mathfrak p}(n)
  =\sum_{j\ge1}\mathbf 1_{q_{\mathfrak p}^j\mid n}.
\end{equation}
Consequently, every finite place is represented by a cyclotomic tower of orbit lengths
$q_{\mathfrak p},q_{\mathfrak p}^2,\ldots$.
\end{proposition}

\begin{proof}
A $\mathfrak p$-ordering of $R$ is obtained by choosing compatible representatives of the residue classes modulo $\mathfrak p^j$ and reading their indices in base $q_{\mathfrak p}$.  The usual Legendre counting argument gives $w_{q_{\mathfrak p}}(n)$.  Taking first differences gives \eqref{eq:polya-recurrence}.
\end{proof}

Grouping prime ideals by norm, set
\[
  \Pi_q(R)=\prod_{\Nm\mathfrak p=q}\mathfrak p.
\]
Then
\[
  [n]!_{R,R}=\prod_{q\ge2}\Pi_q(R)^{w_q(n)}.
\]
This formula makes the class-group content explicit: the factorial line is scalarizable for all $n$ precisely when the relevant products $\Pi_q(R)$ are principal.  For number fields their classes generate the P\'olya--Ostrowski group \cite{Leriche2011}.

\subsubsection{A principal example: the Gaussian integers}
Let $R=\bbZ[i]$.  The prime $2$ is ramified, primes $p\equiv1\pmod4$ split, and primes $p\equiv3\pmod4$ are inert.  If $\mathfrak p_2=(1+i)$, then
\begin{equation}\label{eq:gaussian-factorial}
  [n]!_{\bbZ[i]}
  =\mathfrak p_2^{w_2(n)}
  \prod_{p\equiv1(4)}(p)^{w_p(n)}
  \prod_{p\equiv3(4)}(p)^{w_{p^2}(n)}.
\end{equation}
Since $\bbZ[i]$ is principal, one may choose the generator
\[
  G_i(n)=(1+i)^{w_2(n)}
  \prod_{p\equiv1(4)}p^{w_p(n)}
  \prod_{p\equiv3(4)}p^{w_{p^2}(n)}.
\]
The first factorial lines are represented by
\[
  1,\ 1,\ 1+i,\ 1+i,\ (1+i)^3,\ 5(1+i)^3,\ 5(1+i)^4,\ldots.
\]
Thus even for a principal ring the scalar sequence is visibly different from $n!$; the splitting type of rational primes has entered the recurrence.

\subsubsection{A genuinely line-valued example}
Let $R=\bbZ[\sqrt{-5}]$.  The prime $2$ ramifies as
\[
  (2)=\mathfrak p_2^2,
  \qquad
  \mathfrak p_2=(2,1+\sqrt{-5}),
  \qquad
  \Nm\mathfrak p_2=2.
\]
Therefore
\[
  [2]!_R=\mathfrak p_2.
\]
The ideal $\mathfrak p_2$ is not principal: a generator would have algebraic norm $2$, whereas $a^2+5b^2=2$ has no integral solution.  Hence the Gamma value at $z=3$ cannot be represented canonically by an element of $K$; it is intrinsically the nontrivial ideal line $\mathfrak p_2$.  This elementary example already forces the line-valued definition.

\subsubsection{Norm recurrence and Dedekind zeta}
For a number field $K$, define
\[
  b_K(n)=\log\Nm_{K/\bbQ}\!\left(
  [n]!_{\cO_K}[n-1]!_{\cO_K}^{-1}\right).
\]
By \eqref{eq:polya-recurrence},
\[
  b_K(n)=\sum_{\mathfrak p}\log q_{\mathfrak p}
  \sum_{j\ge1}\mathbf 1_{q_{\mathfrak p}^j\mid n}.
\]
Consequently, for $\Re(s)>1$,
\begin{equation}\label{eq:number-field-dirichlet}
  \sum_{n\ge1}\frac{b_K(n)}{n^s}
  =\zeta(s)\left(-\frac{\zeta_K'}{\zeta_K}(s)\right).
\end{equation}
If $\gamma_K$ is the Euler--Kronecker constant of $K$, then the Laurent expansion at $s=1$ is
\begin{equation}\label{eq:number-field-laurent}
  \frac1{(s-1)^2}+\frac{\gamma-\gamma_K}{s-1}+O(1).
\end{equation}
Lamoureux proved the resulting number-field Stirling formula
\begin{equation}\label{eq:number-field-stirling}
  \log\Nm([n]!_{\cO_K})
  =n\log n+(\gamma-\gamma_K-1)n+o_K(n)
\end{equation}
for every number field $K$ \cite{Lamoureux2014}.  The transform \eqref{eq:number-field-dirichlet} gives a compact explanation of the two displayed coefficients: the double pole yields $n\log n-n$, while the coefficient of $(s-1)^{-1}$ supplies $(\gamma-\gamma_K)n$.  Lamoureux's thesis goes further than \eqref{eq:number-field-stirling}, deriving a refined explicit formula whose remainder contains a sum over the nontrivial zeros of $\zeta_K$.  In the present framework those zero terms should be interpreted as spectral contributions of the centered norm determinant rather than as part of the primitive Gamma skeleton.

For $K=\bbQ$ one has $\gamma_K=\gamma$, and \eqref{eq:number-field-stirling} reduces to the first two terms of the ordinary Stirling formula.

\subsubsection{Optimal energies and the Bhargava--Barnes lift}
Let $T=\{x_0,\ldots,x_n\}\subset\cO_K$ and define its energy ideal by
\begin{equation}\label{eq:optimal-energy}
  \mathcal E(T)
  =\left(\prod_{i\ne j}(x_i-x_j)\right).
\end{equation}
This is the square of the Vandermonde ideal, up to the harmless sign arising from reversing ordered pairs.  Byszewski, Fr\k{a}czyk, and Szumowicz proved that $T$ is $n$-optimal precisely when
\begin{equation}\label{eq:optimal-energy-factorials}
  \mathcal E(T)
  =\left(\prod_{m=1}^{n}[m]!_{\cO_K}\right)^2,
\end{equation}
and in that case this ideal divides the energy of every other $(n+1)$-point subset of $\cO_K$ \cite{ByszewskiFraczykSzumowicz2017}.

The identity has a direct Gamma-theoretic interpretation.  Let $\mathscr G_K$ be the Bhargava Gamma line of $S=\cO_K$, and define its second, or Barnes, lift by
\begin{equation}\label{eq:bhargava-barnes-recurrence}
  \mathscr B_K(1)\cong\cO_K,
  \qquad
  \mathscr B_K(z+1)
  \cong\mathscr G_K(z)\otimes\mathscr B_K(z).
\end{equation}
This is the line-valued analogue of the Barnes recurrence $G(z+1)=\Gamma(z)G(z)$ \cite{Barnes1901}.  Iterating \eqref{eq:bhargava-barnes-recurrence} gives
\begin{equation}\label{eq:bhargava-barnes-integers}
  \mathscr B_K(n+2)
  \cong\bigotimes_{m=1}^{n}\mathscr G_K(m+1)
  \cong\prod_{m=1}^{n}[m]!_{\cO_K}.
\end{equation}
Hence an $n$-optimal set satisfies the exact determinant-line identity
\begin{equation}\label{eq:energy-barnes}
  \boxed{\ \mathcal E(T)\cong\mathscr B_K(n+2)^{\otimes2}.\ }
\end{equation}
The square has a concrete combinatorial source: the energy uses ordered pairs $i\ne j$, whereas the Vandermonde determinant uses each unordered pair once.

Combining \eqref{eq:energy-barnes} with Lamoureux's formula yields the Barnes-level Stirling law
\begin{equation}\label{eq:optimal-energy-stirling}
  \log\Nm_{K/\bbQ}\mathcal E(T)
  =n^2\log n
   -\left(\frac32+\gamma_K-\gamma\right)n^2
   +o_K(n^2),
\end{equation}
which is the energy asymptotic obtained in \cite{ByszewskiFraczykSzumowicz2017}.  Indeed,
\[
  \log\Nm\mathcal E(T)
  =2\sum_{m=1}^{n}\log\Nm([m]!_{\cO_K}),
\]
and summing \eqref{eq:number-field-stirling} uses
\[
  \sum_{m\le n}m\log m
  =\frac12n^2\log n-\frac14n^2+O(n\log n).
\]
Even when an $n$-optimal set does not exist, the line $\mathscr B_K(n+2)^{\otimes2}$ remains the formal optimal-energy determinant and the divisibility statement gives the corresponding universal norm lower bound.

More generally, define the one-variable Bhargava--Barnes hierarchy by
\begin{equation}\label{eq:iterated-bhargava-gamma}
  \mathscr G_K^{(1)}=\mathscr G_K,
  \qquad
  \mathscr G_K^{(r+1)}(1)\cong\cO_K,
  \qquad
  \mathscr G_K^{(r+1)}(z+1)
  \cong\mathscr G_K^{(r)}(z)\otimes
        \mathscr G_K^{(r+1)}(z).
\end{equation}
Here the superscript labels iteration, not tensor power.  This hierarchy is distinct from Bhargava's multivariable factorials: it repeatedly integrates the same one-dimensional determinant transport, whereas the multivariable theory changes the underlying Vandermonde geometry.

\subsubsection{The Carlitz factorial}
Now let $R=A=\bbF_q[t]$.  Write
\[
  n=\sum_{i\ge0}n_iq^i,
  \qquad 0\le n_i<q,
\]
and let
\[
  D_0=1,
  \qquad
  D_i=\prod_{\substack{a\in A\ \mathrm{monic}\\ \deg a=i}}a
  =(t^{q^i}-t)D_{i-1}^{q}.
\]
Bhargava's simultaneous ordering is obtained by reading the coefficient vector of a polynomial as the base-$q$ expansion of its index, and its factorial is the Carlitz product
\begin{equation}\label{eq:carlitz-factorial}
  \Pi_A(n)=[n]!_{A,A}=\prod_{i\ge0}D_i^{n_i}.
\end{equation}
This is the classical Carlitz factorial \cite{Bhargava2000,Carlitz1938,Thakur1991}.
For a monic irreducible $P$ of degree $d$,
\begin{equation}\label{eq:carlitz-local}
  v_P(\Pi_A(n))
  =\sum_{j\ge1}\left\lfloor\frac{n}{q^{dj}}\right\rfloor,
\end{equation}
which is exactly \eqref{eq:polya-factorial} because $\Nm P=q^d$.

The digit carry gives an explicit recurrence.  If $v_q(n)$ denotes the exponent of $q$ dividing the integer index $n$, then
\begin{equation}\label{eq:carlitz-recurrence}
  \frac{\Pi_A(n)}{\Pi_A(n-1)}
  =\prod_{j=1}^{v_q(n)}(t^{q^j}-t),
\end{equation}
with the empty product equal to $1$.  For $q=2$, if $D_1=t^2-t$ and $D_2=(t^4-t)D_1^2$, the first values are
\[
  1,\ 1,\ D_1,\ D_1,\ D_2,\ D_2,\ D_2D_1,\ D_2D_1,\ldots.
\]

There is also an exact digital Stirling formula.  Since $\deg D_i=iq^i$,
\begin{equation}\label{eq:carlitz-degree}
  \deg\Pi_A(n)=\sum_{i\ge0}i n_iq^i.
\end{equation}
If $m=\lfloor\log_q n\rfloor$, then
\begin{equation}\label{eq:carlitz-digital-stirling}
  \deg\Pi_A(n)
  =mn-\sum_{i=0}^{m-1}(m-i)n_iq^i
  =n\log_q n+O_q(n).
\end{equation}
At the two extreme digital phases,
\[
  \deg\Pi_A(q^m)=mq^m,
\]
and
\[
  \deg\Pi_A(q^m-1)
  =\frac{(m-1)q^{m+1}-mq^m+q}{q-1}.
\]
Thus the function-field Stirling remainder is naturally digit-dependent and can have linear-size log-periodic fluctuations.  Goss and Thakur constructed infinite- and finite-place interpolations of these factorials, including reflection and multiplication analogues \cite{Goss1988,Thakur1991}.  In the present language, those functions are scalar trivializations of the place-completed Carlitz determinant line.

\section{The cubic-power set}

Let
\[
  C=\{m^3:m\in\bbZ\}.
\]
Write
\[
  \nu_p(n)=v_p(n!_C),
  \qquad
  a_{C,p}(n)=\nu_p(n)-\nu_p(n-1).
\]
Fares and Johnson developed the characteristic sequences and $p$-orderings of $d$th powers \cite{FaresJohnson2012}; Takeda emphasizes that the higher-degree global asymptotic problem splits according to the residue of $p$ modulo $d$ \cite{Takeda2023}.

\subsection{The cubic-residue sector}

For $p\equiv1\pmod3$, the known local formula is
\begin{equation}\label{eq:cube-split-local}
  \nu_p(n)=v_p((3n)!).
\end{equation}
The complete-block digit statistics are generated by a three-state carry matrix.

Write $p=3h+1$ and define
\[
  A_p(y)=\sum_{j=0}^{h}y^{3j},\qquad
  B_p(y)=y^2\sum_{j=0}^{h-1}y^{3j},\qquad
  C_p(y)=y\sum_{j=0}^{h-1}y^{3j}.
\]
Set
\begin{equation}\label{eq:carry-matrix}
  T_p(y)=
  \begin{pmatrix}
    A_p(y)&B_p(y)&C_p(y)\\
    C_p(y)&A_p(y)&B_p(y)\\
    B_p(y)&C_p(y)&A_p(y)
  \end{pmatrix}.
\end{equation}

\begin{proposition}[Exact carry determinant]\label{prop:carry-det}
Let $v(y)=(1,y,y^2)^t$ and $e_0=(1,0,0)^t$.  Then
\[
  F_{p,k}(y):=e_0^tT_p(y)^kv(y)
  =\sum_{0\le n<p^k}y^{s_p(3n)},
\]
and
\begin{equation}\label{eq:carry-resolvent}
  \sum_{k\ge0}F_{p,k}(y)z^k
  =\frac{(1-z)^2}{\det(I-zT_p(y))}.
\end{equation}
Moreover,
\begin{align}\label{eq:carry-det-factor}
  \det(I-zT_p(y))
  &=(1-z(A+B+C))\notag\\
  &\quad\times\bigl(1-z(2A-B-C)
  +z^2(A^2+B^2+C^2-AB-AC-BC)\bigr),
\end{align}
where $A=A_p(y)$, $B=B_p(y)$, and $C=C_p(y)$.
\end{proposition}

\begin{proof}
Multiplication by $3$ in base $p$ has carry states $0,1,2$.  The entries of $T_p(y)$ enumerate input digits by output digit and outgoing carry.  Iterating the automaton gives the first identity.  The resolvent identity follows from
\[
  \sum_{k\ge0}e_0^t(zT_p(y))^kv(y)
  =e_0^t(I-zT_p(y))^{-1}v(y),
\]
and a direct adjugate computation gives the numerator $(1-z)^2$.  Since $T_p(y)$ is circulant, its three Fourier eigenmodes yield \eqref{eq:carry-det-factor}.
\end{proof}

At $y=1$,
\[
  \det(I-zT_p(1))=(1-pz)(1-z)^2,
\]
so the boundary numerator cancels two unit carry modes and leaves the Euler-type factor $(1-pz)^{-1}$.

\begin{corollary}[Complete-block digit sum]\label{cor:digit-sum}
For $p\equiv1\pmod3$,
\[
  \sum_{0\le n<p^k}s_p(3n)
  =\frac{p-1}{2}kp^k+p^k-1.
\]
Consequently,
\[
  \sum_{0\le n<p^k}\nu_p(n)
  =\frac{3p^{2k}-(p-1)kp^k-5p^k+2}{2(p-1)}.
\]
\end{corollary}

\begin{proof}
Differentiate \eqref{eq:carry-resolvent} at $y=1$ and use \eqref{eq:cube-split-local} together with Legendre's digit-sum formula.
\end{proof}

\subsection{Character projection and the Archimedean half-determinant}

Let
\[
  P_3(n)=(3n)(3n-1)(3n-2)
\]
and let $\chi=\chi_{-3}$.  Define
\[
  \ell_\chi(m)=\sum_{p\ne3}\chi(p)v_p(m)\log p.
\]
For $p\equiv2\pmod3$, write
\[
  \delta_p(n)=a_{C,p}(n)-v_p(n).
\]

\begin{proposition}[Exact character decomposition]\label{prop:character-decomp}
The cube recurrence satisfies
\begin{equation}\label{eq:cube-recurrence-decomp}
  \log\frac{n!_C}{(n-1)!_C}
  =\frac12\log\bigl(nP_3(n)\bigr)+\varepsilon_C(n),
\end{equation}
where
\begin{align}\label{eq:finite-error}
  \varepsilon_C(n)
  &=\frac12\bigl(\ell_\chi(P_3(n))-\ell_\chi(n)\bigr)\notag\\
  &\quad+\sum_{p\equiv2\,(3)}\delta_p(n)\log p\notag\\
  &\quad+\left(a_{C,3}(n)-\frac12v_3(nP_3(n))\right)\log3.
\end{align}
\end{proposition}

\begin{proof}
Use \eqref{eq:cube-split-local} in the split sector and apply the two character projectors
\[
  \mathbf 1_{p\equiv1(3)}=\frac{1+\chi(p)}2,
  \qquad
  \mathbf 1_{p\equiv2(3)}=\frac{1-\chi(p)}2
\]
to the prime factorizations of $P_3(n)$ and $n$, keeping the ramified prime $3$ separate.
\end{proof}

The trivial-character multiplier is therefore
\[
  M_{C,\infty}(z)
  =\sqrt{z(3z)(3z-1)(3z-2)}.
\]

\begin{corollary}[Cube Archimedean determinant]\label{cor:cube-arch}
The normalized solution of
\[
  G(z+1)=M_{C,\infty}(z)G(z),
  \qquad G(1)=1,
\]
is
\begin{equation}\label{eq:cube-arch-gamma}
  \Gamma_{C,\infty}(z)
  =\sqrt{\Gamma(z)\Gamma(3z-2)}.
\end{equation}
Equivalently, its rational divisor is
\[
  [0]+\frac12\left[\frac13\right]
  +\frac12\left[\frac23\right].
\]
\end{corollary}

\begin{proof}
The functional equation follows from the ordinary Gamma recurrence.  The divisor statement follows from the triplication formula.
\end{proof}

The denominator $2$ is the denominator of the character idempotents for
\[
  \Gal(\bbQ(\zeta_3)/\bbQ)\simeq C_2.
\]
Thus the scalar half-determinant is naturally a section of a square-root line.  Its square is single-valued:
\[
  \wideGamma_{C,\infty}(z)
  =\Gamma(z)\Gamma(3z-2).
\]

\subsection{The exact finite correction}

Define the integral squared correction cocycle
\begin{equation}\label{eq:cube-dp}
  d_p(n)=2a_{C,p}(n)-v_p(n)-v_p(P_3(n)).
\end{equation}
Then
\begin{equation}\label{eq:cube-relative-integer}
  \Delta_C(N+1)
  :=\prod_p p^{\sum_{n\le N}d_p(n)}
  =\frac{(N!_C)^2}{N!(3N)!}.
\end{equation}

\begin{proposition}[Split-prime cycle towers]\label{prop:split-cycle-towers}
For $p\equiv1\pmod3$,
\[
  d_p(n)
  =\sum_{j\ge1}
  \left(
  \mathbf1_{n\equiv3^{-1}\!\!\pmod{p^j}}
  +\mathbf1_{n\equiv2\cdot3^{-1}\!\!\pmod{p^j}}
  \right).
\]
Hence the split-prime correction is a product of two marked cyclotomic cycle towers.
\end{proposition}

\begin{proof}
By \eqref{eq:cube-split-local},
\[
  d_p(n)=v_p(P_3(n))-v_p(n)
  =v_p(3n-1)+v_p(3n-2),
\]
and each valuation is expanded into congruence indicators modulo $p^j$.
\end{proof}

For $p\equiv2\pmod3$ and $p=3$, the correction sequence is signed and nonlinear.  Exact local computations show no small finite-state compression.  This leads to the central open problem.

\begin{problem}[Inert-prime renewal determinant]\label{prob:inert-det}
Construct, from the intrinsic $p$-ordering renewal system for each $p\equiv2\pmod3$, a normalized relative determinant section $\Delta^0_{C,p}(z)$ interpolating the centered cumulative cocycle of \eqref{eq:cube-dp}.  Prove normal convergence of the renormalized prime product and a uniqueness theorem excluding integer-vanishing entire factors.
\end{problem}

\subsection{The linear constant}

Let $c_p$ denote the local asymptotic slope of $\nu_p(n)$ when it exists.  The renewal calculation predicts
\[
  c_3=\sqrt6-\frac32,
\]
and, for $p\equiv2\pmod3$,
\[
  c_p=\frac{-3+\sqrt{9+\dfrac{12p}{(p-1)^2}}}{2}.
\]
This leads to the following candidate normalization constant.

\begin{conjecture}[Cube finite-place mean]\label{conj:beta}
The limit
\[
  \beta_C
  =\lim_{N\to\infty}
  \frac1N\log\frac{N!_C}{\sqrt{N!(3N)!}}
\]
exists and equals
\begin{align}\label{eq:beta-L-form}
  \beta_C
  &=\left(\sqrt6-\frac52\right)\log3
  -\frac{L'}{L}(1,\chi_{-3})\notag\\
  &\quad+\sum_{p\equiv2\,(3)}
  \left(
  c_p-\frac1{p-1}-\frac2{p^2-1}
  \right)\log p.
\end{align}
The remaining prime sum is absolutely convergent.
\end{conjecture}

The numerical value from primes up to $10^7$ is
\[
  \beta_C\approx-0.7737452033.
\]
The decomposition separates a ramified term, the Galois term $-L'/L(1,\chi_{-3})$, and the nonlinear inert-prime renewal defect.

\begin{figure}[t]
  \centering
  \includegraphics[width=.76\textwidth]{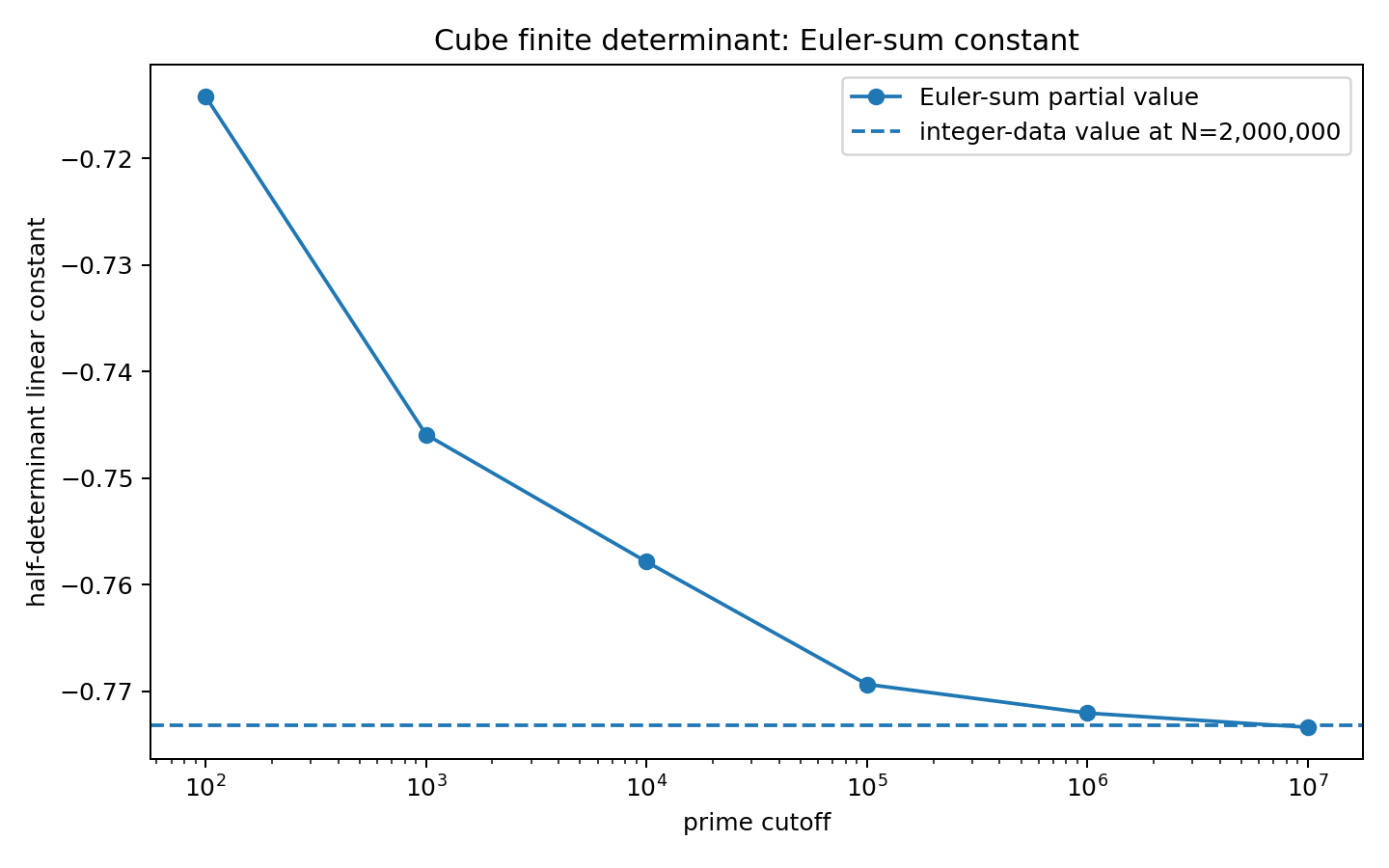}
  \caption{Numerical convergence of the candidate cube finite-place constant.  This figure is evidence for Conjecture~\ref{conj:beta}, not a proof.}
  \label{fig:beta}
\end{figure}

\subsection{Candidate completed cube determinant}

Assuming \cref{prob:inert-det,conj:beta}, the scalar squared cube Gamma should be
\begin{equation}\label{eq:cube-completed}
  \wideGamma_C(z)
  =e^{2\beta_C(z-1)}
  \Gamma(z)\Gamma(3z-2)
  \Delta^0_{C,\mathrm{fin}}(z),
\end{equation}
with
\[
  \wideGamma_C(N+1)=(N!_C)^2.
\]
The Bhargava Gamma object is a normalized section of the square-root line
\[
  \Gamma_C^{\otimes2}=\wideGamma_C.
\]
The positive real lift at the positive integers satisfies $\Gamma_C(N+1)=N!_C$.

\section{Admissibility boundaries}

The existence of $n!_S$ does not imply determinant-admissibility.  We record several possible failures.

\begin{enumerate}[label=(F\arabic*),leftmargin=2.7em]
  \item The local slopes $\nu_{S,\mathfrak p}(n)/n$ may fail to converge.
  \item The local ordering sequence may admit no canonical finite-state, renewal, or regularly varying model.
  \item The primitive infinite-place spectrum may be nondiscrete or nonregularizable.
  \item The rational primitive multiplicities may have unbounded denominators, so no finite $q_S$ clears them.
  \item The centered global determinant product may have no canonical renormalization.
  \item The factorial growth may jump among incompatible scales, preventing a stable Stirling hierarchy.
\end{enumerate}

Rapid growth alone is not the obstruction: geometric progressions grow exponentially and lead to $q$-Gamma determinants.  The obstruction is irregularity of the local place geometry.  A set such as
\[
  \{1,2^2,3^{3^3},4^{4^{4^4}},\ldots\}
\]
is a natural candidate for failure because its successive differences can have unrelated local factorizations and no stable renewal structure.

\section{Number-field and function-field completion}

The one-dimensional number-field and function-field cases belong to the same theory.  The finite-place factorial remains ideal-valued, while the completion is supplied by the selected infinite places.

\subsection{Number fields}

For $R=\cO_K$, the completed Gamma line should be an Arakelov line
\[
  \overline{\mathscr G}_{S,\cO_K}(z)
  =\left(
  \text{finite factorial lattice},
  \text{spectral metrics at }v\mid\infty
  \right).
\]
The Stirling formula is an asymptotic for its Arakelov degree, while the sequence of ideal classes
\[
  [[n]!_{S,\cO_K}]\in\Pic(\cO_K)
\]
records an additional discrete component invisible to the norm.

For the basic pair $S=R=\cO_K$, \eqref{eq:number-field-dirichlet} identifies the norm recurrence with the product of the ordinary zeta factor and the logarithmic derivative of the Dedekind zeta function.  Lamoureux's theorem \eqref{eq:number-field-stirling} fixes the first two metric coefficients, while his zero-explicit refinement shows that the centered determinant retains genuinely global spectral information \cite{Lamoureux2014}.  The pole data \eqref{eq:number-field-laurent} therefore suggest the primitive norm-side completion
\begin{equation}\label{eq:number-field-gamma-skeleton}
  \Gamma_{K,\infty}^{\Nm}(z)
  =e^{(\gamma-\gamma_K)(z-1)}\Gamma(z),
\end{equation}
up to a centered finite determinant whose logarithm is $o(z)$.  This does not trivialize the ideal line: in a non-P\'olya field the fibers still carry nontrivial classes.  Rather, \eqref{eq:number-field-gamma-skeleton} is the scalar metric shadow of the completed line.  The example $K=\bbQ(\sqrt{-5})$, where $[2]!_{\cO_K}$ is nonprincipal, shows that the ideal-class component and the metric component are genuinely independent.

A reflection formula, if it exists, should therefore consist of two simultaneous statements: a duality isomorphism of the Arakelov determinant lines and a scalar reflection identity for their norms.  The former retains unit and class-group information that the latter necessarily forgets.

\subsection{Function fields}

For a global function field with constant field $\bbF_q$, a chosen place or finite set of places at infinity replaces the Archimedean embeddings.  If $A$ is the corresponding ring of functions regular away from infinity, then for $S=A$ the finite factorial line again satisfies \eqref{eq:polya-factorial}.  Put
\[
  b_A(n)=\deg\bigl([n]!_A[n-1]!_A^{-1}\bigr).
\]
Writing
\[
  \zeta_A(s)=\prod_{\mathfrak p}(1-(\Nm\mathfrak p)^{-s})^{-1},
\]
one obtains the exact Dirichlet identity
\begin{equation}\label{eq:function-field-dirichlet}
  \sum_{n\ge1}\frac{b_A(n)}{n^s}
  =\frac{\zeta(s)}{\log q}
  \left(-\frac{\zeta_A'}{\zeta_A}(s)\right).
\end{equation}
Unlike the number-field case, $\zeta_A(s)$ is a rational function of $q^{-s}$; its pole at $s=1$ is repeated periodically at
\[
  s=1+\frac{2\pi i k}{\log q},\qquad k\in\bbZ.
\]
These vertical translates are the spectral source of the log-periodic terms visible in the exact Carlitz formula \eqref{eq:carlitz-digital-stirling}.  Thus the natural function-field Stirling theorem is not merely
$n\log_q n+B n+o(n)$: it generally contains a bounded or linear-size periodic function of $\{\log_q n\}$.

The completed metric is measured by divisor degree or an $\infty$-adic norm, and the infinite-place determinant should recover the known Carlitz--Goss--Thakur Gamma functions after a sign normalization and choice of uniformizer.  Thakur's reflection and multiplication formulas \cite{Thakur1991} provide concrete models for the duality axiom: the classical sine determinant is replaced by a period and a finite product determined by the chosen place at infinity.  For higher-genus function fields, the initial Riemann--Roch irregularity contributes an additional finite boundary determinant, exactly as boundary states occur in the finite carry examples.

\section{Why the multivariable theory is separate}

Bhargava's multivariable factorials replace products of scalar differences by determinants of multivariate evaluation matrices.  The local state space is then governed by monomial filtrations, generalized Vandermonde complexes, Newton polytopes, and the geometry of the Zariski closure of $S\subset R^d$.  The natural infinite-place completion is expected to involve Barnes multiple Gamma functions or higher spectral determinants.  Since this changes the indexing geometry rather than merely the base field, it is deferred to a separate paper.

\section{Proof roadmap}

The present framework becomes a complete theorem once the following tasks are settled.

\begin{enumerate}[label=(P\arabic*),leftmargin=2.7em]
  \item Prove the local renewal slope and centered asymptotic for the inert cube primes uniformly in $p$.
  \item Construct the intrinsic inert-prime relative determinant from the quantile-renewal operator, including its boundary maps.
  \item Establish the trace--arithmetic identity for the determinant and the local correction cocycle.
  \item Prove normal convergence of the centered global product and the identity \eqref{eq:beta-L-form}.
  \item Prove a uniqueness theorem in a natural determinant-growth class.
  \item Determine whether the cube renewal complex has a self-duality yielding a canonical Pfaffian or square-root section.
  \item Verify determinant-admissibility for a nonmonomial cubic polynomial image from Takeda's families.
  \item Develop the metric completion uniformly over number fields and function fields.
\end{enumerate}

\section{Concluding perspective}

The usual Gamma function is the canonical completion of the ordinary factorial recurrence.  Bhargava factorials suggest that the correct generalization is not a single universal scalar special function.  It is a place-completed determinant object attached to local ordering dynamics.  The scalar function, when it exists, is a shadow of this line.

The proposed slogan is
\[
  \boxed{
  \text{Bhargava Gamma}
  =
  \text{determinant-line completion of a factorial calculus}.
  }
\]
For admissible sets, Stirling becomes spectral asymptotics and reflection becomes duality.  For nonadmissible sets, the failure itself records the absence of a coherent local-to-global determinant structure.

\end{document}